\documentclass[9pt]{article}
\usepackage{amsfonts,amssymb,amsmath}
\usepackage{graphicx,graphics,color}

 \makeatletter
 \@addtoreset{equation}{section}
 \makeatother

 \newcounter{enunciato}[section]

 \newtheorem{ittheorem}{Theorem}
 \newtheorem{itlemma}{Lemma}
 \newtheorem{itproposition}{Proposition}
 \newtheorem{itdefinition}{Definition}
 \newtheorem{itcorollary}{Corollary}
 \newtheorem{itconjecture}{Conjecture}

 \newenvironment{theorem}{\addtocounter{enunciato}{1}
 \begin{ittheorem}}{\end{ittheorem}}

 \newenvironment{lemma}{\addtocounter{enunciato}{1}
 \begin{itlemma}}{\end{itlemma}}

 \newenvironment{conjecture}{\addtocounter{enunciato}{1}
 \begin{itconjecture}}{\end{itconjecture}}

\newcommand{\halmos}{\rule{1ex}{1.4ex}}

\newenvironment{proof}{\noindent {\em Proof}.\,\,}
   {\hspace*{\fill}$\halmos$\medskip}

\def \ba {\begin{array}}
\def \ea {\end{array}}

\def \Z {{\mathbb Z}}
\def \R {{\mathbb R}}

\def \N {{\mathbb N}}

\def \P {{\mathbb P}}
\def \E {{\mathbb E}}
\def \T {{\mathbb T}}

\def \da {\downarrow}

\def \cp {\mathrm{cap}\,}

\begin{document}
\title{Heat content and inradius \\
for regions with a Brownian boundary}

\author{\renewcommand{\thefootnote}{\arabic{footnote}}
M.\ van den Berg
\footnotemark[1]
\\
\renewcommand{\thefootnote}{\arabic{footnote}}
E.\ Bolthausen
\footnotemark[2]
\\
\renewcommand{\thefootnote}{\arabic{footnote}}
F.\ den Hollander
\footnotemark[3]
}

\footnotetext[1]{ School of Mathematics, University of Bristol,
University Walk, Bristol BS8 1TW, United Kingdom. }

\footnotetext[2]{
Institut f\"ur Mathematik, Universit\"at Z\"urich,
Winterthurerstrasse 190, CH-8057 Z\"urich, Switzerland.
}

\footnotetext[3]{
Mathematical Institute, Leiden University, P.O.\ Box 9512,
2300 RA Leiden, The Netherlands.
}
\date{12 March 2013}

\maketitle

\begin{abstract}
In this paper we consider $\beta[0,s]$, Brownian motion of time
length $s>0$, in $m$-dimensional Euclidean space $\R^m$ and on the
$m$-dimensional torus $\T^m$. We compute the expectation of (i)
the heat content at time $t$ of $\R^m \backslash \beta [0,s]$ for
fixed $s$ and $m=2,3$ in the limit $t \da 0$, when $\beta[0,s]$ is
kept at temperature $1$ for all $t>0$ and $\R^m \backslash \beta
[0,s]$ has initial temperature $0$, and (ii) the inradius of $\T^m
\backslash \beta [0,s]$ for $m=2,3,\cdots$ in the limit
$s\to\infty$.

\vskip 0.5truecm
\noindent
{\it AMS} 2000 {\it subject classifications.} 35J20, 60G50.\\
{\it Key words and phrases.} Laplacian, Brownian motion, Wiener sausage,
heat content, inradius, spectrum.

\medskip\noindent
{\it Acknowledgment.}
MvdB was supported by The Leverhulme Trust, Research Fellowship 2008/0368,
EB by SNSF-grant 20-100536/1, and FdH by ERC Advanced Grant 267356-VARIS.

\end{abstract}

\newpage


\section{Introduction and main results}
\label{S.1}

Asymptotic properties of the heat content and the inradius for
regions with a fractal boundary have received a lot of attention
in the literature. Most of the focus has been on porous regions
(e.g.\ the $m$-dimensional Euclidean space $\R^m$ from which a
Poisson cloud of non-polar sets is removed \cite{Sz}), and regions
with a fractal polygonal boundary e.g.\ the von Koch snow flake
and its relatives \cite{vdB1, vdB2, vdBedHo99}. In this paper we
consider the region obtained from $\R^m$ or the $m$-dimensional
torus $\T^m$ by cutting out a Brownian path of time length $s$. In
Sections \ref{S.1.1} and \ref{S.1.2} we consider the heat content,
and in Section \ref{S.1.3} the inradius. We formulate some open
problems in Section \ref{S.1.5}. The proofs are deferred to
Sections \ref{S.2}--\ref{S.3}.


\subsection{Heat content outside compact sets}
\label{S.1.1}

Let $K$ be a compact non-polar set in $\R^m$ with boundary $\partial K$, and
let $v\colon\,\R^m \backslash K \times [0,\infty)\to \R$ be the unique weak
solution of the heat equation
\begin{equation}
\label{eq1alt}
\left\{\begin{array}{llll}
\Delta v(x;t) &=& \partial v(x;t)/\partial t, &x\in \R^m \backslash K,\, t>0,\\
v(x;t) &=& 1, &x \in \partial K,\, t>0,\\
v(x;0) &=& 0, &x \in \R^m \backslash K.
\end{array}
\right.
\end{equation}
Then $v(x;t)$ represents the temperature at point $x$ at time $t$
when $\partial K$ is kept at temperature $1$ and the initial
temperature is $0$. The heat content of $\R^m \backslash K$ at
time $t$ is defined by
\begin{equation}
\label{e9}
E_K(t) = \int_{\R^m \backslash K} v(x;t)\,dx.
\end{equation}
If $\partial K$ is $C^\infty$, then $E_K(t)$ has an asymptotic
series expansion for $t\da 0$ of the form
\begin{equation}
\label{EKt0} E_K(t) = \sum_{j=1}^J a_j(K)\,t^{j/2} +
O\big(t^{(J+1)/2}\big), \qquad t \da 0,\,J\in\N,
\end{equation}
where the coefficients are local geometric invariants of $\R^m
\backslash K$. In particular,
\begin{equation*}
\begin{aligned}
a_1(K) &= 2 \pi^{-1/2} \int_{\partial K} dz,\\
a_2(K) &= 2^{-1}(m-1) \int_{\partial K} H(z)\,dz,
\end{aligned}
\end{equation*}
where $dz$ is the surface measure on $\partial K$, and $H(z)$ is
the mean curvature at $z$ of $\partial K$ with inward orientation.
Formulas of this type can be found in the general setting of
Riemannian manifolds and Laplace-type operators \cite{G}. The case
where $\partial K$ is only $C^3$ was settled by probabilistic
tools in \cite{vdBLeG}, and the expansion in \eqref{EKt0} holds for
$J=2$.

The asymptotic behaviour for $t\to \infty$ is different and its analysis does
not require smoothness of $\partial K$. For $m=3$ it is shown in
\cite{LeG1}, \cite{P} (see \cite{S1} for earlier results) that if
$K$ is a compact set, then
\begin{equation}
\label{e10} E_K(t) = \sum_{j=1}^3 b_j(K)\,t^{(3-j)/2} +
O\big(t^{-1/2}\big), \qquad t \to \infty,
\end{equation}
with
\begin{equation}
\label{e10alt}
\begin{aligned}
b_1(K) &= \cp(K),\\
b_2(K) &= 2^{-1}\pi^{-3/2}\cp(K)^2,\\
b_3(K) &= (4\pi)^{-2}\cp(K)^3-|K|-(8\pi)^{-1}\int_K \int_K
\|x-y\|\,\mu_K(dx)\mu_K(dy),\\
\end{aligned}
\end{equation}
where $\|\cdot\|$ is the Euclidean norm, $\mu_K$ is the equilibrium measure of
$K$, $\cp(K) = \int_K \mu_K(dx)$ is the Newtonian capacity of $K$, and $|K|$ is
the Lebesgue measure of $K$. For $m=2$ it is shown in \cite{LeG} that if $K$ is
a non-polar set, then
\begin{equation}
\label{e11} E_K(t)= t \sum_{j=1}^{J} b_j(K)\,(\log t)^{-j} +
O\big(t(\log t)^{-(J+1)}\big), \ \  t \to \infty,\,J \in \N,
\end{equation}
where $b_1(K)=4\pi$, and where the higher-order coefficients all depend on the
logarithmic capacity of $K$ only.

For a wide class of regions with a fractal boundary it is known
that $E_K(t)$ is comparable with $t^{(m-d)/2}$ for $t \da 0$,
where $d$ is the interior Minkowski dimension of the boundary of
$\R^m \backslash K$ \cite{vdB1}. In the case of a self-similar
boundary it is sometimes possible to obtain more detailed results
\cite{vdBedHo99}, \cite{vdB2}. It turns out that if there is a
dominant arithmetic sequence of length scales, then the leading
term is $t^{(m-d)/2}$ times a periodic function of $\log(t^{-1})$.
If there is no such sequence, then the leading term is
$t^{(m-d)/2}$ times a constant. However, neither the periodic
function nor the constant is known explicitly. In Section
\ref{S.1.2} below we study the case where $K$ is a Brownian path
of time length $s$. We will see that for this case there is no
periodic function of $\log(t^{-1})$.


\subsection{Expectation of the heat content outside a Brownian path}
\label{S.1.2}

The solution of (\ref{eq1alt}) is given by
\begin{equation}
\label{vsolrep} v(x;t) = \P_x(\tau_K \leq t),
\end{equation}
where $\tau_K=\inf\{u \geq 0 \colon\,\bar\beta(u) \in K\}$ with $(\bar\beta(u),
\,u \geq 0;\,\P_x,\,x \in \R^m)$ Brownian motion on $\R^m$. Let $E_{\bar\beta[0,s]}
(t)$ denote the heat content of $\R^m\backslash\beta[0,s]$ at time $t$, and let
\begin{equation}
\label{avhc}
E(s,t) = \E_0\big(E_{\beta[0,s]}(t)\big)
\end{equation}
denote its expectation. Since $\cp(\beta[0,s])=0$ for
$m=4,5,\cdots$, only $m=2,3$ are relevant. In Section \ref{S.2} we
will prove the following duality property.

\begin{theorem}
\label{thdual}
\textup{(a)} Let $m=2,3$. Then for $s,t>0$,
\begin{equation}
\label{Escal}
E(s,t) = E(t,s), \qquad E(s,t) = (t/s)^{m/2}\,E(s,s^2/t).
\end{equation}
\textup{(b)} Let $m=3$. Then for $s>0$,
\begin{equation}
\label{e13} E(s,t) = c_1(s)\,t^{1/2} + c_2(s)\,t + c_3(s)\,t^{3/2}
+ O(t^2), \qquad t \da 0,
\end{equation}
with
\begin{equation}
\begin{aligned}
c_1(s) &=\mathcal{C}_1 s,\\
c_2(s) &= 2^{-1}\pi^{-3/2}\mathcal{C}_2 s^{1/2},\\
c_3(s) &= (4\pi)^{-2}\mathcal{C}_3 - (8\pi)^{-1} \E_0\left(\int_{\R^3} \int_{\R^3}
\|x-y\|\,\mu_{\beta[0,1]}(dx)\mu_{\beta[0,1]}(dy)\right),
\end{aligned}
\end{equation}
where
\begin{equation*}
 \mathcal{C}_i=
\E_0\Big(\big(\cp\big(\beta[0,1]\big)\big)^i\Big), \qquad i=1,2,3.
\end{equation*}
\textup{(c)} Let $m=2$. Then for $s>0$,
\begin{equation*}
 E(s,t) = 4\pi s\, \big(\log(t^{-1})\big)^{-1}
+ O\big((\log(t^{-1}))^{-2}\big), \qquad t \da 0.
\end{equation*}
\end{theorem}

Theorem \ref{thdual}(a) states a duality property from which Theorem
\ref{thdual}(b-c) easily follows. Indeed from
(\ref{e10}--\ref{e10alt}) and \eqref{avhc} we obtain that for $m=3$,
\begin{equation}
\label{Etlarge}
E(s,t) = \sum_{j=1}^3 c_j(s)\,t^{(3-j)/2} + O\big(t^{-1/2}\big), \qquad t\to\infty,
\end{equation}
with $c_j(s) = \E_0^1(b_j(\beta[0,s]))$. By combining \eqref{Etlarge} with \eqref{Escal},
we obtain the claim in Theorem \ref{thdual}(b). Similarly, it follows from (\ref{e11})
that for $m=2$,
\begin{equation}
\label{Etlargealt}
E(s,t)= 4\pi\,t(\log t)^{-1} + O\big(t(\log t)^{-2}\big), \qquad t \to \infty.
\end{equation}
By combining \eqref{Etlargealt} with \eqref{Escal}, we obtain the claim in Theorem
\ref{thdual}(c).


\subsection{Expectation of the inradius of a torus cut by a Brownian path}
\label{S.1.3}

Let $\T^m$ denote the $m$-dimensional torus, which we identify with the set
$(-\frac{1}{2},+\frac{1}{2}]^m$. We denote the distance on $\T^m$ between
points $x$ and $y$ by $d(x,y)$. $\T^m$ is a compact and connected Riemannian
manifold without boundary. The associated Laplace-Beltrami operator is the
generator of Brownian motion on $\T^m$. The latter process can be obtained by
wrapping Brownian motion on $\R^m$ around $\T^m$. Namely, let $(\tilde{\beta}(u),
\,u\geq 0;\,\P_x,\,x\in \R^m)$ denote Brownian motion on $\R^m$, with the
Laplacian as generator, and put
\begin{equation*}
\beta(u) =
\big(\tilde{\beta}(u)+(\tfrac12,\cdots,\tfrac12)\big)\,(\mathrm{mod}\,\Z^m)
- (\tfrac12,\cdots,\tfrac12).
\end{equation*}
Then $(\beta(u),u\geq 0;\,\P_x,\,x\in \R^m)$ is Brownian motion on $\T^m$.

On $\T^m$ we define a random distance function $d_s$ by putting
\begin{equation*}
d_s(x) = \inf\{d(x,y)\colon\,y\in\beta[0,s]\}, \qquad x\in\T^m,
\end{equation*}
i.e., the distance of $x$ to $\beta[0,s]$. The inradius of
$\T^m\backslash\beta[0,s]$ is the random variable $\rho(s)$
defined by
\begin{equation*}
 \rho(s) = \sup\{d_s(x)\colon\,x\in\T^m\}.
\end{equation*}
The supremum is attained because $d_s$ is continuous and $\T^m$ is compact.
Hence there exists an open ball with radius $\rho(s)$ in $\T^m\backslash\beta[0,s]$.
The inradius is a non-trivial random variable in all dimensions.

\begin{theorem}
\label{inradmgeq3}
If $m=3,4,\cdots$, then
\begin{equation*}
\lim_{s\to\infty}\,\left(\frac{s}{\log s}\right)^{1/(m-2)}\,
\E_0(\rho(s))= \left(\frac{m}{(m-2)\kappa_m}\right)^{1/(m-2)},
\end{equation*}
where $\kappa_m$ is the Newtonian capacity of the ball with radius
$1$ in $\R^m$, given by
\begin{equation*}
\label{e5.20} \kappa_m=4\pi^{m/2}[\Gamma((m-2)/2)]^{-1}.
\end{equation*}
\end{theorem}

\begin{theorem}
\label{inradm=2} If $m=2$, then
\begin{equation*}
\lim_{s\to\infty} s^{-1/2}\,\log\E_0(\rho(s))=-\pi^{1/2}.
\end{equation*}
\end{theorem}



\subsection{Discussion and open problems}
\label{S.1.5}

{\bf Heat content.} Since the Minkowski dimension of the Brownian
path equals $d=2$, the power $1/2$ in (\ref{e13}) for $m=3$ agrees
with the exponent $(m-d)/2$ mentioned below (\ref{e11}). The
absence of a periodic function multiplying the term $t^{1/2}$
reflects the fact that there is no dominant arithmetic sequence of
length scales in the Brownian path. Comparing (\ref{e13}) with
(\ref{e10}--\ref{e10alt}), we see that $\beta[0,s]$ has an
effective area $2^{-1}\pi^{1/2}\mathcal{C}_1 s$ and an effective
mean curvature integral $2^{-1}\pi^{-3/2}\mathcal{C}_2 s^{1/2}$.

We see from Theorem \ref{thdual} that the complement of the
Brownian path heats up much faster for $m=2$ than for $m=3$ as $t$
increases from $0$. With the help of (\ref{e11}) it is actually
possible to obtain the full asymptotic series for $m=2$. The
geometry of $\beta[0,1]$ enters into this series only via the
expectation of the powers of the logarithm of the logarithmic
capacity of $\beta[0,1]$.

\medskip\noindent
{\bf Inradius.} Theorems \ref{inradmgeq3}--\ref{inradm=2} identify
the scaling behavior of the expectation of the inradius. Strong
laws of large numbers were derived in \cite{DPR} and \cite{DPRZ}.
Our proofs for the upper bounds in Section 3.1 are based on Wiener
sausage estimates and spectral decomposition, while the lower
bounds in Section 3.2 and Section 3.3 respectively,  are based on
results of \cite{DPR} and \cite{DPRZ} which rely in turn on
excursions of Brownian motions.

\medskip\noindent
{\bf Spectrum.} There are several other set functions that are
closely related to the inradius such as  the principal Dirichlet
eigenvalue. For $s>0$, let $\Delta_s$ be the Laplacian acting in
$L^2(\T^m \backslash \beta[0,s])$ with Dirichlet boundary
condition on $\beta[0,s]$. The spectrum of $-\Delta_s$ is bounded
from below by the spectrum of $-\Delta$ on $\T^m$. Since the
latter is discrete, the spectrum of $-\Delta_s$ is discrete, with
eigenvalues $(\lambda_{j,s})_{j\in\N}$, labeled in non-decreasing
order and including multiplicities. Since the Newtonian capacity
of $\beta[0,s]$ is zero for $m>3$, the spectrum of $-\Delta_s$ is
non-trivial if and only if $m=2,3$. Since
$s\mapsto\T^m\backslash\beta[0,s]$ is decreasing we have, by
domain monotonicity of Dirichlet eigenvalues \cite{RS}, that
$s\mapsto \lambda_{j,s}$ is increasing.

\begin{conjecture}
\label{specm=2} If $m=2$, then
\begin{equation}
\label{e17} \lim_{s\to\infty} s^{-1/2}\log \E_0(\lambda_{1,s}) =
2\pi ^{1/2}.
\end{equation}
\end{conjecture}

\begin{conjecture}
\label{specm=3} If $m=3$, then
\begin{equation}\label{e16}
\lim_{s\rightarrow \infty} \left(\frac{\log s}{s}\right)^2
\E_0(\lambda_{1,s}) = 9^{-1}(2\pi)^4.
\end{equation}
\end{conjecture}

The heuristic behind Conjecture \ref{specm=2} is as follows. The
complement of the Brownian path $\beta[0,s]$ in $\T^2$ consists of
simply connected open components. Hence the spectrum of the
Dirichlet Laplacian acting in $L^2(\T^2\setminus\beta[0,s])$ is
the union of the spectra of the Dirichlet Laplacian for all the
components. In particular, $\lambda_{1,s}$ is the minimum over all
first eigenvalues of these components. For a simply connected open
set $\Omega$ in $\R^2$ with inradius $\rho$ we know from
\cite{An86} that the first Dirichlet eigenvalue is comparable to
$\rho^{-2}$. The same is not true for the bottom of the spectrum
of a simply connected open set $\Omega$ in $\T^2$. Indeed, the
inradius of $\T^2$ is bounded from above by half its diameter.
Hence, if $\Omega$ is almost all of $\T^2$, then the bottom of the
spectrum is close to $0$ while $\rho^{-2}$ is bounded away from
$0$. However, for large $s$ it is very unlikely that such a large
component exists. Thus, we have that the typical component for
large $s$ has small $\rho(s)$, and hence, by \cite{An86},
\begin{equation*}
\lambda_{1,s}\asymp \rho(s)^{-2}.
\end{equation*}
The right-hand side is of order $e^{2(\pi s)^{1/2}}$, which
explains (\ref{e17}).

The heuristic behind Conjecture \ref{specm=3} is as follows. The
simplifying features for $m=2$ are absent for $m=3$. We expect
that with high probability the largest open ball in
$\T^3\setminus\beta[0,s]$ with inradius $\rho(s)$ is ``densely
surrounded'' by the Brownian path. So $\lambda_{1,s}$ is the first
Dirichlet eigenvalue of a ball with radius $\rho(s)$ in $\T^3$ (or
$\R^3$). Hence
\begin{equation}
\label{e18a*} \lambda_{1,s}\approx \pi^2\rho(s)^{-2}.
\end{equation}
We also expect that the inradius $\rho(s)$ gets very narrowly
distributed around its mean as $s\to\infty$. Hence for large $s$,
\begin{equation*}
\E_0(\lambda_{1,s})\approx \pi^2(\E_0(\rho(s)))^{-2}.
\end{equation*}
The asymptotic behaviour of the latter expectation can be read off
from Theorem \ref{inradmgeq3} for $m=3$, and implies \eqref{e16}.

\medskip\noindent
{\bf Large deviations.} As explained at the end of Section
\ref{S.2}, it is easy to show that for $m=3$ and $s>0$ the
following strong law of large numbers holds:
\begin{equation}
\label{eq11} \lim_{t \da 0} E_{\beta[0,s]}(t)/E(s,t) = 1 \qquad
\P_0-a.s.
\end{equation}
We expect that both $\rho(s)$ and $\lambda_{1,s}$ satisfy the
strong law of large numbers as $s\to\infty$. It is interesting to
determine their large deviation behaviour. For $m \geq 3$ this was
achieved in \cite{GdH}, which was inspired by an unpublished
earlier version of the present paper. Interestingly, the large
deviations are so uncostly in one direction that they do not imply
our results about the expectation.


\section{Proof of Theorem \ref{thdual}}
\label{S.2}

Let
\begin{equation*}
\big(\beta_i(u),\,u \geq 0;\,\P^i_x,\,x \in \R^m\big), \qquad i=1,2,
\end{equation*}
be two independent Brownian motions on $\R^m$. Recalling (\ref{vsolrep}), we have
from (\ref{e9}) that
\begin{equation*}
E_K(t) + |K| = \E^2_0\big(|W_{\beta_2}^K(t)|\big),
\end{equation*}
where
\begin{equation*}
W_{\beta_2}^K(t)= \cup_{u \in [0,t]} \{K + \beta_2(u)\} = K + \beta_2[0,t]
\end{equation*}
is the $K$-set Wiener sausage associated with $\beta_2$ up to time $t$. Now choose
$K=\beta_1[0,s]$, $s>0$. Then, since $|\beta_1[0,s]|=0$, the heat content in $\R^m
\backslash\beta_1[0,s]$ becomes
\begin{equation*}
E_{\beta_1[0,s]}(t) = \E_0^2\big(|W(s,t)|\big)
\end{equation*}
with
\begin{equation*}
W(s,t) = W_{\beta_2}^{\beta_1[0,s]}(t) = \beta_1[0,s]+\beta_2[0,t]
= W_{\beta_1}^{\beta_2[0,t]}(s).
\end{equation*}
The expected heat content becomes
\begin{equation*}
E(s,t) = \E_0^1\big(E_{\beta_1[0,s]}(t)\big) = (\E_0^1 \otimes \E_0^2)(|W(s,t)|)
= \E_0^2\big(E_{\beta_2[0,t]}(s)\big).
\end{equation*}

We have the following two elementary lemmas.

\begin{lemma}
\label{Lem1}
For all $s,t\geq 0$,
\begin{equation*}
E(s,t) = E(t,s).
\end{equation*}
\end{lemma}

\begin{proof}
Note that
\begin{equation}
\label{e20}
E(s,t) = (\E_0^1\otimes\E_0^2)\big(|W(s,t)|\big) = \int_{\R^m} dx\,\,
(\P_x^1 \otimes \P_0^2)\big(\tau_{\beta_2[0,t]} \leq s\big)
\end{equation}
with
\begin{equation*}
\tau_{\beta_2[0,t]} = \inf\big\{u \geq 0\colon\,\beta_1(u)\in\beta_2[0,t]\big\}.
\end{equation*}
We may rewrite (\ref{e20}) as
\begin{equation*}
E(s,t) = \int_{\R^m} dx\,\,(\P_x^1 \otimes \P_0^2)\big(\beta_1[0,s]
\cap \beta_2[0,t] \neq \emptyset\big),
\end{equation*}
from which the symmetry property follows via the change of variable $x \to -x$.
\end{proof}

\begin{lemma}
\label{Lem2}
For all $a>0$ and $s,t \geq 0$,
\begin{equation*}
|W(s,t)| \triangleq a^{-m}|W(a^2s,a^2t)|,
\end{equation*}
where $\triangleq$ denotes equality in distribution.
\end{lemma}

\begin{proof}
Let $W_{\beta_1}^A(t)$ is the $A$-set Wiener sausage associated
with $\beta_1$ up to time $t$, and note the two scaling relations
\begin{equation*}
\beta_2[0,t] \triangleq a^{-1}\,\beta_2[0,a^2t],\qquad
W_{\beta_1}^A(s) \triangleq a^{-d}\,W_{\beta_1}^{aA}(a^2s).
\end{equation*}
The claim follows by choosing
$A=a^{-1}\beta_2[0,a^2s]$.
\end{proof}

Theorem \ref{thdual}(a) follows from Lemma \ref{Lem1} and Lemma \ref{Lem2} with
$a^2=s/t$. Theorems \ref{thdual}(b-c) follow from (\ref{Escal}),
(\ref{Etlarge}--\ref{Etlargealt}) and the scaling relations
\begin{equation*}
\begin{aligned}
\cp\big(\beta_1[0,s]\big) &\triangleq s^{1/2}\,\cp\big(\beta_1[0,1]\big),\\
\mu_{\beta_1[0,s]}(sA) &\triangleq \mu_{\beta_1[0,1]}(A),
\end{aligned}
\end{equation*}
which are valid for any compact subset $A\subset \R^3$. The latter imply that $c_j(s)
= s^{j/2}c_j(1)$, $s>0$, for $m=3$.

A standard renewal argument \cite{S1}, \cite{S2} gives the strong law of large numbers
for $|W(s,t)|$, namely, for $s>0$,
\begin{equation*}
\lim_{t\to\infty} |W(s,t)|/\E^1_0\big(|W(s,t)|\big) = 1
\qquad (\P_0^1\otimes\P_0^2)-a.s.
\end{equation*}
This in turn implies that for $s>0$,
\begin{equation*}
\lim_{t\to\infty} \E^2_0\big(|W(s,t)|\big)/(\E^1_0\otimes\E_0^2)\big(|W(s,t)|\big)
= 1  \qquad \P_0^1-a.s.,
\end{equation*}
which is the same as
\begin{equation}
\label{e30alt}
\lim_{t\to\infty} E_{\beta_1[0,s]}(t)/E(s,t) = 1  \qquad \P_0^1-a.s.
\end{equation}
The claim in (\ref{eq11}) follows from (\ref{e30alt}) via Lemma \ref{Lem2} with
$a^2=s/t$.


\section{Proof of Theorems \ref{inradmgeq3}--\ref{inradm=2}}
\label{S.3}

For $x\in\T^m$ and $\epsilon>0$, let $T_{x,\epsilon}=\inf\{u\geq
0\colon\,\beta(u)\in B_x(\epsilon)\}$, where $(\beta(u),\,u\geq
0;\,\P_x,\,x\in\T^m)$ is Brownian motion on $\T^m$, and
$B_x(\epsilon)$ is the open ball with center $x$ and radius
$\epsilon$ in $\T^m$. Then
\begin{equation*}
 T_\epsilon = \sup_{x\in\T^m} T_{x,\epsilon}
\end{equation*}
is the cover time of $\T^m$ by the Wiener sausage with radius
$\epsilon$. By translation invariance, we have
\begin{equation*}
 \P_0[T_{\epsilon}>s] = \P_x[T_{\epsilon}>s], \ \ \
x\in\T^m,\,s\geq 0,
\end{equation*}
which, since $|\T^m|=1$, gives
\begin{equation}
\label{e5.2}
\P_0[T_{\epsilon}>s] = \int_{\T^m} dx\,\P_x[T_{\epsilon}>s].
\end{equation}


\subsection{Upper bound}

Let $N\in\N$, and let $\{x_1,x_2,\cdots,x_{N^m}\}=(N^{-1}\Z)^m\cap\T^m.$
Let $\eta \in (0,1/4)$ be arbitrary, and consider the collection of open balls
with centers $\{x_1,x_2,\cdots,x_{N^m}\}$ and radii
$(1-\eta)\epsilon$. There exists $v\in \{x_1,x_2,\cdots,x_{N^m}\}$
such that $d(x,v)\leq (2N)^{-1}m^{1/2}$. For $N\geq
m^{1/2}/(2\epsilon \eta),$ we have $B_x(\epsilon) \supset
B_v((1-\eta)\epsilon)$. This implies that if $N\ge
m^{1/2}/(2\epsilon \eta)$, then
\begin{equation*}
\big\{B_{x_i}((1-\eta)\epsilon) \cap \beta[0,s]\ne
\emptyset, i=1,2,\cdots,N^m\big\} \subset \{T_{\epsilon}\le s\}.
\end{equation*}
It follows that
\begin{align}
\label{e5.4}
\P_x[T_{\epsilon}>s] &\leq
1-\P_x\big[B_{x_i}((1-\eta)\epsilon)\cap
\beta[0,s]\ne \emptyset, i=1,\cdots,N^m\big]\nonumber\\
&\leq
1-\E_x\left[\prod_{i=1}^{N^m}\left(1-1_{B_{x_i}((1-\eta)\epsilon)\cap
\beta[0,s]= \emptyset}\right)\right]\nonumber\\
&\leq
\left(\sum_{i=1}^{N^m}\P_x\big[{B_{x_i}((1-\eta)\epsilon)\cap
\beta[0,s]= \emptyset}\big]\right)\wedge 1.
\end{align}
By \eqref{e5.2} and \eqref{e5.4}, we have
\begin{align}
\label{e5.5} \P_0[T_{\epsilon}>s] &\leq \int_{\T^m} dx
\left(\sum_{i=1}^{N^m}\P_x\big[{B_{x_i}((1-\eta)\epsilon)\cap
\beta[0,s]= \emptyset}\big]\right)\wedge 1\nonumber\\
&\leq \left(N^ m\int_{\T^m}
dx\,\P_x[{B_{x_1}((1-\eta)\epsilon)\cap \beta[0,s]=
\emptyset}]\right)\wedge1.
\end{align}

Next, let $\mu_{1,(1-\eta)\epsilon}<\mu_{2,(1-\eta)\epsilon}\leq\cdots$
be the spectrum of the Dirichlet Laplacian acting in
$L^2(\T^m\setminus \overline{B}_{x_1} ((1-\eta)\epsilon))$, with a
corresponding orthonormal set of eigenfunctions
$\{\psi_{j,(1-\eta)\epsilon}, j=1,2,\cdots\}$. Then
\begin{equation}
\label{e5.6}
\P_x\big[B_{x_1}((1-\eta)\epsilon)\cap
\beta[0,s]=\emptyset\big] = \sum_{j=1}^{\infty}
e^{-s\mu_{j,(1-\eta)\epsilon}}\, \psi_{j,(1-\eta)\epsilon }(x)
\int_{\T^m\setminus\overline{B}_{x_1}((1-\eta)\epsilon))}dy\,
\psi_{j,(1-\eta)\epsilon}(y).
\end{equation}
Hence, by Parseval's identity and \eqref{e5.6},
\begin{align}
\label{e5.7}
\int_{\T^m}dx\,\P_x[B_{x_1}((1-\eta)\epsilon)\cap\beta[0,s]=\emptyset]
&= \sum_{j=1}^{\infty} e^{-s\mu_{j,(1-\eta)\epsilon}}
\left(\int_{\T^m\setminus\overline{B}_{x_1}((1-\eta)\epsilon))}\,
\psi_{j,(1-\eta)\epsilon}\right)^2\nonumber\\
&\leq e^{-s\mu_{1,(1-\eta)\epsilon}}
\sum_{j=1}^{\infty}\left(\int_{\T^m\setminus\overline{B}_{x_1}((1-\eta)\epsilon))}\,
\psi_{j,(1-\eta)\epsilon}\right)^2\nonumber\\
&=
e^{-s\mu_{1,(1-\eta)\epsilon}}\,(|\T^m|-|\overline{B}_{x_1}((1-\eta)\epsilon)|)
\nonumber\\
&\leq e^{-s\mu_{1,(1-\eta)\epsilon}}.
\end{align}
By \eqref{e5.5} and \eqref{e5.7},
\begin{equation*}
\P_0[T_{\epsilon}>s] \leq \left(N^m e^{-s\mu_{1,(1-\eta)\epsilon}}\right)\wedge 1.
\end{equation*}
Since $\textup{diam}(\T^m)=2^{-1}m^{1/2}$, the inradius is bounded
from above by $4^{-1}m^{1/2}$. Moreover,
$\{\rho(s)>\epsilon\}=\{T_{\epsilon}> s\}.$ Hence
\begin{equation}
\label{e5.9}
\E_0(\rho(s))=\int_0^{4^{-1}m^{1/2}}
d\epsilon\,\P_0[T_{\epsilon} > s] \leq \int_0^{4^{-1}m^{1/2}}
d\epsilon\,
\left(\left(N^me^{-s\mu_{1,(1-\eta)\epsilon}}\right)\wedge
1\right).
\end{equation}
Let
\begin{equation}
\label{e5.10}
N=[m^{1/2}/(2\epsilon \eta)]+1.
\end{equation}
Since $\epsilon \le 4^{-1}m^{1/2}$ and $\eta< 1/4$, we have
\begin{equation}
\label{e5.11}
N\leq m^{1/2}/(\epsilon \eta).
\end{equation}

\medskip\noindent
$\bullet$
First consider the case $m=2$. By \cite{O}, we have that
\begin{equation*}
\mu_{1,\epsilon} = 2\pi \left(\log(1/{\epsilon})\right)^{-1} +
O((\log(1/{\epsilon}))^{-2}),\ \  \epsilon \downarrow 0.
\end{equation*}
Hence there exists $\epsilon_0(\eta)$ such that, for $\epsilon\leq
\epsilon_0(\eta)$,
\begin{equation*}
\mu_{1,\epsilon} \geq
2\pi\left(\log(1/{\epsilon})\right)^{-1}(1-\eta).
\end{equation*}
So, abbreviating $\epsilon \le \epsilon_1(\eta)=(2^{1/2}/4)\wedge
\epsilon_0(\eta)$, we have for $\epsilon \leq \epsilon_1(\eta)$,
\begin{align}
\label{e5.14}
\mu_{1,(1-\eta)\epsilon} &\geq 2\pi
\left(\log(1/{(1-\eta)\epsilon})\right)^{-1}\nonumber\\
&= 2\pi \left(\log(1/{\epsilon})\right)^{-1}
\left(1+\frac{\log(1/(1-\eta))}{\log(1/\epsilon)}\right)^{-1}(1-\eta)\nonumber\\
&\geq 2\pi \left(\log(1/{\epsilon})\right)^{-1}
\left(1+\frac{2\log(1/(1-\eta))}{3\log 2}\right)^{-1}(1-\eta)\nonumber\\
&\geq 2\pi \left(\log(1/{\epsilon})\right)^{-1}
\left(1-\log(1/(1-\eta))\right)(1-\eta)\nonumber\\
&\geq 2\pi \left(\log(1/{\epsilon})\right)^{-1}
\left(1-\frac{\eta}{1-\eta}\right)(1-\eta)\nonumber\\
&= 2\pi \left(\log(1/{\epsilon})\right)^{-1}(1-2\eta).
\end{align}
Putting \eqref{e5.9}, \eqref{e5.11} and \eqref{e5.14} together, we
obtain
\begin{align}
\label{e5.15}
\E_0(\rho(s)) &\leq
\int_0^{\epsilon_1(\eta)}d\epsilon \left((2(\eta\epsilon)^{-2}
e^{-2\pi s(\log(1/\epsilon))^{-1}(1-2\eta)})\wedge 1\right)\nonumber\\
&\qquad + 2\int_{\epsilon_1(\eta)}^{\infty} d\epsilon\,(\eta
\epsilon)^{-2} e^{-2\pi s(\log(1/\epsilon_1(\eta)))^{-1}(1-2\eta)}.
\end{align}
The second term in \eqref{e5.15} is bounded from above by
\begin{equation*}
2\eta^{-2}\epsilon_1(\eta)^{-1}e^{-(\pi/2)
s(\log(1/\epsilon_1(\eta)))^{-1}}.
\end{equation*}
By changing variables, $\epsilon=e^{-\theta}$, we obtain that the
first integral is bounded from above by
\begin{align}
\label{e5.17}
&2\eta^{-2}\int_0^{\infty} d\theta\,e^{-\theta}
\left(e^{2\theta-2\pi s(1-2\eta)/{\theta}}\wedge 1\right)\nonumber\\
&\leq 2\eta^{-2}\int_0^{(\pi s(1-2\eta))^{1/2}} d\theta\,
e^{\theta-2\pi s(1-2\eta)/{\theta}}+2\eta^{-2}
\int_{(\pi s(1-2\eta))^{1/2}}^{\infty} d\theta\,e^{-\theta}\nonumber\\
&\leq 2\eta^{-2}((\pi s)^{1/2}+1)e^{-(\pi s(1-2\eta))^{1/2}}.
\end{align}
It follows from (\ref{e5.15}--\ref{e5.17}) that for
$\eta\in(0,1/4)$,
\begin{equation*}
\limsup_{s\to \infty}s^{-1/2}\log \E_0(\rho(s)) \leq
-(\pi(1-2\eta))^{1/2}.
\end{equation*}
This proves the upper bound in Theorem \ref{inradm=2} for $m=2$
because $\eta\in(0,1/4)$ was arbitrary.

\medskip\noindent
$\bullet$
Next consider the case $m=3,4,\cdots$. By Theorem 1 in
\cite{CF}, we have that
\begin{equation*}
\mu_{1,\epsilon}=\kappa_m\epsilon^{m-2}(1+o(1)), \qquad \epsilon \downarrow 0.
\end{equation*}
Hence there exists $\epsilon_0(\eta)$ such that, for $\epsilon\leq
\epsilon_0(\eta)$,
\begin{equation}
\label{e5.21}
\mu_{1,\epsilon} \geq
\kappa_m\epsilon^{m-2}(1-\eta).
\end{equation}
By \eqref{e5.9}, \eqref{e5.11} and \eqref{e5.21} we have, for any
$\eta \in (0,1/4)$,
\begin{align}
\label{e5.22}
\E_0(\rho(s)) &\leq
\int_0^{\min\{m^{1/2}/4,\epsilon_0(\eta)\}} d\epsilon
\left(m^{m/2}(\epsilon \eta)^{-m}
e^{-s\kappa_m(1-\eta)\epsilon^{m-2}}\wedge 1\right)\nonumber\\
&\qquad +\int_{\min\{m^{1/2}/4,\epsilon_0(\eta)\}}^{m^{1/2}/4}
d\epsilon\,
m^{m/2}(\epsilon\eta)^{-m}e^{-s\kappa_m(1-\eta)\epsilon_0(\eta)^{m-2}}.
\end{align}
The second term in \eqref{e5.22} is bounded from above by
\begin{equation}
\label{e5.23}
m^{m/2}(m-1)^{-1}\eta^{-m}(\max\{4/m^{1/2},\epsilon_0(\eta)^{-1}\})^{m-1}
e^{-s\kappa_m(1-\eta)\epsilon_0(\eta)^{m-2}}.
\end{equation}
The first term in \eqref{e5.22} is bounded from above by
\begin{align}
\label{e5.24}
&\int_0^{\infty}d\epsilon\left(m^{m/2}(\epsilon\eta)^{-m}
e^{-s\kappa_m(1-\eta)\epsilon^{m-2}}\wedge1\right)\nonumber\\
&\qquad =
(1-\eta)^{-1/(m-2)}(m-2)^{-1}(s\kappa_m)^{-1/(m-2)}\int_0^{\infty}
d\theta\, \theta^{(3-m)/(m-2)}(K\theta^{-m/(m-2)}e^{-\theta}
\wedge 1),
\end{align}
where
\begin{equation}
\label{e5.25}
K=m^{m/2}\eta^{-m}(1-\eta)^{m/(m-2)}(s\kappa_m)^{m/(m-2)}.
\end{equation}
Let $\theta_K$ be the unique positive root of
\begin{equation}
\label{e5.26}
K\theta^{-m/(m-2)}e^{-\theta}=1.
\end{equation}
The integral in the right-hand side of \eqref{e5.24} equals
\begin{align}
\label{e5.27}
&\int_0^{\theta_K} d\theta\,\theta^{(3-m)/(m-2)}
+ K\int_{\theta_K}^{\infty} d\theta\,\theta^{(3-2m)/(m-2)}e^{-\theta}\nonumber\\
&\qquad \leq (m-2)\theta_K^{1/(m-2)}+K\theta_K^{(3-2m)/(m-2)}e^{-\theta_K}\nonumber\\
&\qquad =(m-2)\theta_K^{1/(m-2)}+\theta_K^{(3-m)/(m-2)}.
\end{align}
For $K\ge e$, we have $\theta_K\ge 1$, and, by \eqref{e5.26},
\begin{equation}
\label{e5.28} e^{\theta_K}= K \theta_K^{-m/(m-2)}\leq K, \qquad
K\geq e.
\end{equation}
Hence $\theta_K\le \log K$ for $K\geq e$. It follows from
(\ref{e5.22}-\ref{e5.28}) that for $s\to \infty$,
\begin{equation*}
\E_0(\rho(s))\le(1-\eta)^{-1/(m-2)}(\log K)^{1/(m-2)}
(s\kappa_m)^{-1/(m-2)}+O(s^{-1/(m-2)}(\log s)^{(3-m)/(m-2)}).
\end{equation*}
Hence
\begin{equation*}
\limsup_{s\to \infty}\,\left(\frac{s}{\log s}\right)^{1/(m-2)} \E_0(\rho(s))
\leq (1-\eta)^{-1/(m-2)}\left(\frac{m}{(m-2)\kappa_m}\right)^{1/(m-2)}.
\end{equation*}
Let $\eta \downarrow 0$ to get the upper bound in Theorem \ref{inradmgeq3}.


\subsection{Lower bound for $m\ge3$}

To prove the lower bound in Theorem \ref{inradmgeq3} we use the
following inequality in \cite {DPR}. Let $\eta \in (0,1/10]$ and
$\delta \in (0,1/10]$ be arbitrary, and let, for $n \in \N$,
\begin{equation}
\label{e5.31}
\epsilon_n=(1-\eta)^n,
\end{equation}
\begin{equation}
\label{e5.32}
v_n=(1-\eta)\kappa_m^{-1}\epsilon_n^{2-m},
\end{equation}
and
\begin{equation}
\label{e5.33}
K_n\ge \epsilon_n^{-m(1-2\delta)}.
\end{equation}
The very last inequality of \cite{DPR} implies that for the
sequence $(\epsilon_n)$ there exists $c=c(\eta,\delta)$ such that
\begin{equation*}
\P_0[T_{\epsilon_n}\le(1-2\eta)v_n \log K_n]\le 4(1-\eta)^{cn}.
\end{equation*}
Let $\phi \in (0,1/4]$ be arbitrary. There exists
$N(\eta,\delta,\phi)\in \N$ such that for $n\ge
N(\eta,\delta,\phi)$,
\begin{equation*}
\P_0[T_{\epsilon_n}\le(1-2\eta)v_n \log K_n]\le \phi.
\end{equation*}
or
\begin{equation*}
\P_0[T_{\epsilon_n}\ge(1-2\eta)v_n \log K_n]\ge 1-\phi.
\end{equation*}
This, together with (\ref{e5.32}--\ref{e5.33}), gives that
\begin{equation}
\label{e5.37}
\P_0[T_{\epsilon_n}\ge C m\kappa_m^{-1}\epsilon_n^{2-m} \log
(\epsilon_n^{-1})]\ge 1-\phi,
\end{equation}
where $C=(1-2\eta)(1-\eta)(1-2\delta)$. We now choose
$n=n(s,\eta)\in \Z$ such that
\begin{equation}
\label{e5.38}
(1-\eta)^{n-1}\ge \left(C(m-2)^{-1}m\kappa_m^{-1}\frac{\log
s}{s}\right)^{1/(m-2)}\ge(1-\eta)^n.
\end{equation}
Then $n\in \N$ and $n\ge N(\eta,\delta,\phi)$ for all $s$ large
enough. By \eqref{e5.31} and \eqref{e5.38}
\begin{equation}\label{e5.39}
\epsilon_n^{2-m}=(1-\eta)^{n(2-m)}\ge
C^{-1}(m-2)m^{-1}\kappa_m\frac{s}{\log s}.
\end{equation}
On the other hand, by \eqref{e5.31} and \eqref{e5.38} we have that
\begin{equation}
\label{e5.40}
\log(\epsilon_n^{-1})\ge
(m-2)^{-1}\log\left(C^{-1}(m-2)m^{-1}\kappa_m\frac{s}{\log
s}\right).
\end{equation}
By (\ref{e5.39}--\ref{e5.40}),
\begin{equation}
\label{e5.41}
C m\kappa_m^{-1}\epsilon_n^{2-m} \log (\epsilon_n^{-1}) \ge s+h(s),
\end{equation}
where
\begin{equation*}
h(s)=\frac{s}{\log s}\log \left(\frac{(m-2)m^{-1}\kappa_m}{\log s}\right).
\end{equation*}
By the definition of $n$ in \eqref{e5.38} and by \eqref{e5.37} and \eqref{e5.41},
we have that for all $s$ sufficiently large,
\begin{equation*}
\P_0[T_{\epsilon_n}\ge s+h(s)]\ge 1-\phi.
\end{equation*}
Hence, by the first equality in \eqref{e5.10}, \eqref{e5.31} and \eqref{e5.38},
\begin{align}
\label{e5.44}
\E_0(\rho(s+h(s)))&=\int_0^{4^{-1}m^{1/2}}
d\epsilon\,\P_0[T_{\epsilon}
> s+h(s)]\nonumber \\ &\ge\int_0^{\epsilon_n}
d\epsilon\,\P_0[T_{\epsilon}
> s+h(s)]\nonumber \\ &\ge (1-\phi)\epsilon_n\nonumber \\ &
\ge (1-\phi)(1-\eta)\ \left(C(m-2)^{-1}m\kappa_m^{-1}\frac{\log
s}{s}\right)^{1/(m-2)}.
\end{align}
It remains to show that \eqref{e5.44} implies the lower bound in
Theorem \ref{inradmgeq3}. We abbreviate $\sigma=s+h(s)$. Since
$(\log \log s)/ \log s\rightarrow 0$ as $s \rightarrow \infty$, we
have that $|h(s)|\le s/2$ for all $s$ sufficiently large. Hence
$\sigma \ge s/2$ for all such $s$, and $s=\sigma-h(s)\le
\sigma-h(2\sigma)$. It follows that for all such $s$,
\begin{equation*}
\E_0(\rho(\sigma))\ge
(1-\phi)(1-\eta)\left(C(m-2)^{-1}m\kappa_m^{-1}\frac{\log
\sigma}{\sigma-h(2\sigma)} \right)^{1/(m-2)}.
\end{equation*}
In particular, it follows that
\begin{align}\label{e5.46}
\liminf_{\sigma \rightarrow \infty}&\
\E_0(\rho(\sigma))\left(\frac{\sigma}{\log
\sigma}\right)^{1/(m-2)}\nonumber \\ &
\ge(1-\phi)(1-\eta)\left(C(m-2)^{-1}m\kappa_m^{-1}\liminf_{\sigma
\rightarrow
\infty}\frac{\sigma}{\sigma-h(2\sigma)}\right)^{1/(m-2)}\nonumber
\\ &
=(1-\phi)(1-\eta)\left(C(m-2)^{-1}m\kappa_m^{-1}\right)^{1/(m-2)}.
\end{align}
Letting first $\phi \downarrow 0$, then $\delta \downarrow 0$ and finally
$\eta \downarrow 0$, we conclude from \eqref{e5.46} that
\begin{equation*}
\liminf_{\sigma \rightarrow \infty}
\E_0(\rho(\sigma))\left(\frac{\sigma}{\log
\sigma}\right)^{1/(m-2)} \geq \left((m-2)^{-1}m\kappa_m^{-1}\right)^{1/(m-2)}.
\end{equation*}
This proves the lower bound in Theorem \ref{inradmgeq3}.


\subsection{Lower bound for $m=2$}

To prove the lower bound in Theorem \ref{inradm=2} we use the
following inequality in \cite {DPRZ}. Let $\delta \in (0,1/10]$ be
arbitrary, fix $\gamma \in (0,1-\delta)$ and let
\begin{equation*}
\epsilon_n=2n^{\gamma-1}.
\end{equation*}
It was shown in \cite{DPRZ} that there exist $N_0(\gamma,\delta)\in \N$
such that for all $n\ge N_0(\gamma,\delta)$,
\begin{equation*}
\P_0[T_{\epsilon_n}\ge \pi^{-1}(1-\gamma-\delta)^2(\log n)^2]\ge
1-\delta.
\end{equation*}
We let $n=n(s,\gamma,\delta)\in\{2,3,\cdots\}$ be such that
\begin{equation}
\label{e5.50}
\pi^{-1}(1-\gamma-\delta)^2(\log n)^2\ge s \ge
\pi^{-1}(1-\gamma-\delta)^2(\log (n-1))^2.
\end{equation}
It follows that, for all $s$ sufficiently large and $n\ge N_0(\gamma,\delta)$,
\begin{equation*}
\P_0[T_{\epsilon_n}\ge s]\ge 1-\delta.
\end{equation*}
In particular, for all $s$ sufficiently large we have that
\begin{align}
\label{e5.52}
\E_0(\rho(s))&=\int_0^{4^{-1}\sqrt2} d\epsilon\,
\P_0[T_{\epsilon}\ge s] \ge \int_0^{\epsilon_n} d\epsilon\,
\P_0[T_{\epsilon}\ge s]\nonumber \\ &\ge\int_0^{\epsilon_n}
d\epsilon\,\P_0[T_{\epsilon_n}\ge s]\ge \epsilon_n(1-\delta)
= 2n^{\gamma-1}(1-\delta).
\end{align}
By the second inequality in \eqref{e5.50}, we have that
\begin{equation}
\label{e5.53}
n\le 1+e^{(\pi s)^{1/2}/(1-\gamma-\delta)}.
\end{equation}
Since $1\le e^{(\pi s)^{1/2}/(1-\gamma-\delta)}$, we have by
(\ref{e5.52}--\ref{e5.53}) that
\begin{equation*}
\E_0(\rho(s))\ge 2^{\gamma}(1-\delta)e^{(\gamma-1)(\pi s)^{1/2}/(1-\gamma-\delta)}.
\end{equation*}
Hence
\begin{equation}
\label{e5.55} \liminf_{s \rightarrow \infty}s^{-1/2}\log
\E_0(\rho(s))\ge(\gamma-1)\pi ^{1/2}/(1-\gamma-\delta)
\end{equation}
Letting $\delta \downarrow 0$ we obtain from \eqref{e5.55} that
\begin{equation*}
\liminf_{s \rightarrow \infty}s^{-1/2}\log \E_0(\rho(s))\ge
-\pi^{-1/2}.
\end{equation*}
This proves the lower bound in Theorem \ref{inradm=2}.


\end{document}